\documentclass[12pt]{amsart}
\usepackage{amsmath}
\usepackage{amsthm}
\usepackage{amssymb}
\usepackage{amsfonts}
\usepackage{bm}
\usepackage[shortlabels]{enumitem}
\usepackage[bookmarks=true,hyperindex,pdftex,colorlinks,citecolor=red,linkcolor=blue]{hyperref}
\usepackage{centernot}
\usepackage{marginnote}
\usepackage{bbm}
\usepackage[all]{xy}
\usepackage{tikz}
\usepackage{comment}
\usetikzlibrary{arrows}
\usetikzlibrary{shapes,decorations}
\usetikzlibrary{positioning}
\usepackage{xcolor}
\usepackage{bigints}
\usepackage{cancel}
\usepackage[
a4paper,                 
left=2.4cm,
right=2.4cm,
top=2.3cm,
bottom=2.5cm,
headheight=14pt,         
includeheadfoot
]{geometry}

\theoremstyle{plain}
\newtheorem{thm}{Theorem}[section]
\newtheorem{prop}[thm]{Proposition}
\newtheorem{lemma}[thm]{Lemma}
\newtheorem{cor}[thm]{Corollary}

\theoremstyle{remark}
\newtheorem{remark}[thm]{Remark}

\theoremstyle{definition}

\newcommand{\N}{\mathbb{N}}
\newcommand{\R}{\mathbb{R}}
\newcommand{\Z}{\mathbb{Z}}
\newcommand{\T}{\mathbb{T}}
\newcommand{\one}{\mathbbm{1}}

\newcommand{\norm}[1]{\left\Vert #1\right\Vert}

\definecolor{darkgreen}{rgb}{.2,.6,.2}

\begin{document}
	
	\title[Polynomial gaps for Kreiss bounded semigroups and operators]
	{Polynomial gaps below linear growth for Kreiss bounded semigroups and operators}
	\author[L. Arnold]{Loris Arnold}

	\address[L. Arnold]{Normandie Univ, UNICAEN, CNRS, LMNO, 14000 Caen, France}
	\email{lfj.arld@gmail.com}
	
	\begin{abstract} We prove that every Kreiss bounded $C_0$-semigroup $(T_t)_{t\ge0}$ on a Hilbert space satisfies \[ \norm{T_t}\le C(1+t)^{1-\varepsilon_K}, \qquad t\ge0, \] where $\varepsilon_K>0$ depends explicitly only on the Kreiss constant. The same conclusion is obtained for positive Kreiss bounded $C_0$-semigroups on $L^p$-spaces, $1<p<\infty$, and in discrete time for Kreiss bounded operators on Hilbert spaces and positive Kreiss bounded operators on $L^p$-spaces. We further obtain a non-quantitative polynomial gap for individually eventually positive Kreiss bounded $C_0$-semigroups on $L^p$-spaces. Finally we prove that every Kreiss bounded operator on a UMD Banach space has a polynomial gap below linear growth\end{abstract}
	\maketitle

	\section{Introduction}
	
	A classical problem in operator theory and in the asymptotic analysis of evolution equations is to determine to what extent first-order resolvent estimates control the growth of the corresponding semigroup or of the powers of an operator. In this paper, we consider the continuous and discrete Kreiss conditions, both on Hilbert spaces and, under positivity assumptions, on $L^p$-spaces.
	
	\subsection*{Background and motivation}
	
	On Hilbert spaces, it was first established that Kreiss boundedness of a $C_0$-semigroup $(T_t)_{t\ge 0}$ implies the linear bound $
	\|T_t\| = O(t)$ (see \cite{EisnerZwart2006,RozendaalVeraar2018}). This estimate was subsequently improved in \cite{Arnold2022} to
	\[
	\|T_t\|
	=
	O\left(\frac{t}{\sqrt{\log(t+1)}}\right).
	\]
	Related estimates for uniformly eventually positive Kreiss bounded semigroups on Banach
	lattices were obtained in \cite{ArnoldCoine2023}. In particular, on
	$L^p$-spaces, $1<p<\infty$, the logarithmic improvement
	\[
	\|T_t\|
	=
	O\left(
	\frac{t}{(\log(t+1))^{\max\{1/p,1/p'\}}}
	\right)
	\]
	was proved, whereas on $(AL)$- and $(AM)$-spaces a genuine polynomial
	improvement
	\[
	\|T_t\|=O(t^{1-\varepsilon})
	\]
	for some $\varepsilon>0$ was already obtained.
	
	In discrete time, and in contrast to the continuous-time setting, Kreiss boundedness of an operator $T$ on any Banach space $X$ always implies the linear bound $
	\|T^n\| = O(n).$ In the Hilbert space setting, this estimate was improved by Cohen, Cuny, Eisner and Lin \cite{CohenCunyEisnerLin2020} to 
	\[
	\|T^n\| = O\!\left(\frac{n}{\sqrt{\log(n+1)}}\right),
	\]
	while Cuny, in \cite{CunyLp}, obtained, for Kreiss bounded operators on UMD spaces, the estimate
	\[
	\|T^n\| = O\!\left(\frac{n}{(\log(n+1))^{1/s}}\right),
	\]
	where $s := \min(q,q^*)$ and $q$, $q^*$ denote the cotypes of $X$ and $X^*$ respectively.
	There is, however, a strong obstruction to any uniform polynomial improvement. Indeed, Eisner and Zwart constructed, for every $\gamma\in(0,1)$, a Kreiss bounded $C_0$-semigroup on a Hilbert space such that
	\[
	\norm{T_t}\gtrsim t^\gamma
	\qquad (t\to\infty);
	\]
	see \cite[Example 4.4]{EisnerZwart2006}. Thus, there is no $\varepsilon>0$ such that
	\[
	\norm{T_t}=O(t^{1-\varepsilon})
	\]
	for every Kreiss bounded $C_0$-semigroup on a Hilbert space. The same phenomenon occurs in discrete time: Bonilla and M\"uller showed that no universal polynomial gap below linear growth can hold, even for uniformly Kreiss bounded operators; see \cite{BonillaMuller2021}.
	
	The purpose of this paper is to show that a genuine polynomial gap nevertheless exists in the Kreiss-bounded setting. For operators and semigroups on Hilbert spaces, and for positive operators and semigroups on $L^p$-spaces,, we obtain explicit quantitative bounds in discrete and continuous time, with exponents depending solely on the Kreiss constant and space parameters. In contrast, for individually eventually positive semigroups, we establish a non-quantitative continuous-time analogue. Finally, we extend our results to Kreiss-bounded operators on UMD spaces, with bounds involving the geometric constants of the underlying space.
	
	The quantitative results obtained in Sections \ref{secContinuous} and
	\ref{secDiscrete} are based on the same self-improvement mechanism. A
	Kreiss-type resolvent estimate first yields a triangular estimate. Applied to
	a suitably normalized reversed orbit, this gives a reciprocal-weight
	inequality that produces a fixed gain when the scale is doubled. Iterating
	this gain and using an upper orbit estimate yields the polynomial improvement.
	In the Hilbertian setting, the triangular estimate follows from a
	Fourier--Plancherel multiplier argument, whereas in the positive $L^p$
	setting it follows from the positive convolution principle of Weis and its
	discrete counterpart.
	
	Section \ref{secUMD} implements the same general principle in Fourier norms.
	The UMD property provides the required triangular estimate, while lower
	Fourier decompositions turn the comparison of consecutive dyadic blocks into
	a multiplicative gain. Iteration again yields a polynomial gap. The
	individually eventually positive case requires a separate argument, since the
	absence of a uniform positivity time prevents a direct application of the
	positive convolution method; the resulting polynomial gap is therefore
	non-quantitative.

	\subsection*{Notation and main results}
	
	Throughout the paper, $H$ denotes a complex Hilbert space and $(\Omega, \Sigma, \mu)$ a measure space. For $1 < p < \infty$, we write $L^p(\Omega) = L^p(\Omega, \Sigma, \mu)$ (or simply $L^p$ when convenient). Its positive cone is
	\[
	L^p(\Omega,\mu)_+
	=
	\{f\in L^p(\Omega,\mu;\mathbb R): f\ge0 \text{ a.e.}\}.
	\] Moreover, we denote by $p'$ the conjugate exponent of $p$.
	
	For $f \in L^2(\R;H)$, we define its Fourier transform $\widehat f \in L^2(\R;H)$ (initially on $L^1\cap L^2$) by
	\[
	\widehat f(\beta) := \int_{\R} e^{-i\beta s} f(s) \, ds, \qquad \beta \in \R.
	\]
	Similarly, for a sequence $f = (f_k)_{k\in\Z} \in \ell^2(\Z;H)$, its Fourier transform $\widehat f \in L^2(\T;H)$ is defined by
	\[
	\widehat f(\gamma) := \sum_{k\in\Z} f_k \gamma^{-k}, \qquad \gamma \in \T,
	\]
	where $\T = \{z \in \mathbb{C} : |z| = 1\}$ carries the normalized Haar measure.
	
	If $B$ is a closed operator, we write
	\[
	R(\lambda,B):=(\lambda I-B)^{-1},\qquad \lambda\in\rho(B),
	\]
	for its resolvent. We write $B(X)$ for the algebra of bounded linear operators on a Banach space $X$.
	
	In continuous time, $(T_t)_{t\ge 0}$ will always denote a $C_0$-semigroup with generator $-A$. Its Kreiss constant and Abel constant are respectively defined by
	\begin{equation*}\label{eqDefKreissContinuous}
		C_K(T)
		:=
		\sup_{\substack{r>0\\ \beta\in\R}}
		r\,\norm{R(r+i\beta,-A)}
		\qquad\text{and}\qquad
		C_A(T)
		:=
		\sup_{r>0}
		r\,\norm{R(r,-A)}.
	\end{equation*}
	The semigroup $(T_t)_{t\ge 0}$ is said to be \emph{Kreiss bounded} when $C_K(T)<\infty$ and \emph{Abel bounded} when $C_A(T)<\infty$. Notice that $C_A(T)\le C_K(T)$, so that Kreiss boundedness implies Abel boundedness. Moreover, when $(T_t)_{t\ge 0}$ acts positively on an $L^p$-space, the reverse implication holds (see \cite[Proposition 2.6]{ArnoldCoine2023}).
	
	In discrete time, for $T\in B(X)$, we set
	\begin{equation*}\label{eqDefKreissDiscrete}
		C_K(T)
		:=
		\sup_{|\lambda|>1}
		(|\lambda|-1)\norm{R(\lambda,T)}
		\qquad\text{and}\qquad
		C_A(T)
		:=
		\sup_{\rho>1}
		(\rho-1)\norm{R(\rho,T)}.
	\end{equation*}
	
	We use the same notation $C_K(T)$ and $C_A(T)$ whether $T$ represents a continuous semigroup or a discrete-time operator, as there is no risk of confusion between the two. The operator $T$ is said to be \emph{Kreiss bounded} if $C_K(T)<\infty$ and \emph{Abel bounded} if $C_A(T)<\infty$. As before, one always has $C_A(T)\le C_K(T)$. Furthermore, if $T$ is positive on an $L^p$-space, then Abel boundedness implies Kreiss boundedness (see \cite[Proposition 5.13]{CohenCunyEisnerLin2020}).
	
	For $r>1$ and $C>0$, it will be convenient to introduce the exponent
	\begin{equation*}\label{eq:def-eta}
		\varepsilon_r(C)
		:=
		\frac1r
		\log_2\!\left(
		1+(2eC)^{-r}
		\right)>0.
	\end{equation*}
	With this notation, our main results take a particularly simple form.
	
	\begin{thm}\label{thmContinuousMain}
		
		\begin{enumerate}[(i)]
			\item
			If $(T_t)_{t\ge 0}$ is a Kreiss bounded $C_0$-semigroup on a Hilbert space $H$ then there exists $C>0$ such that
			\begin{equation}\label{eqMainHilbertContinuous}
				\norm{T_t}
				\le
				C(1+t)^{1-\varepsilon_2(C_K(T))},
				\qquad t\ge 0.
			\end{equation}
			
			\item
			Let $1<p<\infty$. If $(T_t)_{t\ge 0}$ is a positive Kreiss bounded $C_0$-semigroup on $L^p(\Omega)$, then there exists $C>0$ such that
			\begin{equation*}\label{eqMainPositiveContinuous}
				\norm{T_t}
				\le
				C(1+t)^{1-\varepsilon_q(C_A(T))},
				\qquad t\ge 0,
			\end{equation*}
			where $q:=p\wedge p'$.
		\end{enumerate}
	\end{thm}
	
	When positivity holds only individually and eventually, we establish : 
	\begin{thm}\label{thmIndEv}
		Let $1<p<\infty$. If $(T_t)_{t\ge 0}$ is an individually eventually positive Kreiss bounded $C_0$-semigroup on $L^p(\Omega)$, then there exist $C>0$ and $\varepsilon >0$, such that
		\begin{equation*}\label{eqMainindPositiveContinuous}
			\norm{T_t}
			\le
			C(1+t)^{1-\varepsilon},
			\qquad t\ge 0.
		\end{equation*}

	\end{thm}
	
	The discrete analogue is as follows.
	
	\begin{thm}\label{thmDiscreteMain}
		\begin{enumerate}[(i)]
			\item
			If $T$ is Kreiss bounded on a complex Hilbert space $H$, then there exists $C>0$ such that
			\begin{equation*}\label{eqMainDiscreteH}
				\norm{T^N}
				\le
				C(N+1)^{1-\varepsilon_2(C_K(T))},
				\qquad N\ge 0.
			\end{equation*}
			
			\item
			Let $1<p<\infty$. If $T$ is Kreiss bounded and positive on $L^p(\Omega)$, then there exists $C>0$ such that
			\begin{equation*}\label{eqMainPositiveDiscrete}
				\norm{T^N}
				\le
				C(N+1)^{1-\varepsilon_q(C_A(T))},
				\qquad N\ge 0.
			\end{equation*}
		\end{enumerate}
	\end{thm}
	In the more general setting of UMD spaces, we obtain the following:
	\begin{thm}\label{thmUMD}
		If $T$ is Kreiss bounded on a UMD space $X$, then there
		exist $C>0$ and $\varepsilon>0$  depending only on $C_K(T)$ and the geometry of $X$ such that
		\begin{equation*}\label{eqMainUMDDiscrete}
			\norm{T^N}
			\le
			C(N+1)^{1-\varepsilon},
			\qquad N\ge0.
		\end{equation*}
		
	\end{thm}
	
	\subsection*{Organization of the paper}
	
	Section \ref{sec:prelim} collects the continuous and discrete preliminary estimates used throughout the paper. In particular, we record the triangular estimates and the upper orbit estimates in both the Hilbertian and positive $L^p$ settings. Section \ref{secContinuous} contains the continuous-time self-improvement argument and the proof of Theorem \ref{thmContinuousMain}. Section \ref{secDiscrete} develops the discrete counterpart and proves the  two assertions of Theorem \ref{thmDiscreteMain}. Section \ref{secIndividual} extends the continuous-time result, in a non-quantitative form, to individually eventually positive semigroups on $L^p$-spaces. Section \ref{secUMD} then treats general Kreiss bounded operators on UMD spaces. Finally, Appendix \ref{appTriangular} gathers the Fourier multiplier arguments underlying the triangular estimates in the Hilbertian setting.
	
	\section{Preliminaries}\label{sec:prelim}
	We begin by stating a positive convolution principle that will be useful. It is a 
	special case of the positive convolution theorem of Weis; see
	\cite[Theorem 2]{Weis1998}. The discrete version follows from the same argument.
	\begin{prop}\label{prop:positive-convolution}
		Let $1\le p<\infty$.
		\begin{enumerate}[(i)]
			\item If $u\mapsto K(u)$ is a strongly operator measurable family of positive
			operators on $L^p(\Omega)$ such that, for every $x\in L^p(\Omega)$, the map
			$u\mapsto K(u)x$ is Bochner integrable and
			\[
			\left(\int_\R K(u)\,du \right) x:=\int_{\R}K(u)x\,du
			\]
			defines a bounded operator on  $L^p(\Omega)$ then the convolution
			\[
			(\mathcal S f)(s):=\int_\R K(u)f(s-u)\,du
			\]
			satisfies
			\[
			\norm{\mathcal S}_{B(L^p(\R;L^p(\Omega)))}\le \norm{\int_\R K(u)\,du}_{B(L^p)}.
			\]
			\item If $(K_n)_{n\in\Z}$ is a finitely supported family of positive operators on $L^p(\Omega)$, then
			\[
			(S f)_j:=\sum_{n\in\Z}K_n f_{j-n}
			\]
			satisfies
			\[
			\norm{S}_{B(\ell^p(\Z;L^p(\Omega)))}\le \norm{\sum_{n\in\Z}K_n}_{B(L^p)}.
			\]
		\end{enumerate}
	\end{prop}

	\subsection{Continuous time}\label{sec:prelim-continuous}
	
	Let $(T_t)_{t\ge0}$ be a $C_0$-semigroup with generator $-A$. For
	$m>0$, set
	\[
	I_m:=[0,m],
	\qquad
	J_m:=[m,2m],
	\]
	and define 
	\begin{equation*}\label{eqDefKmContinuous}
		(\mathcal K_mf)(s)
		:=
		\int_{J_m}T_{v-s}f(v)\,dv,
		\qquad s\in I_m.
	\end{equation*}
	
	We now prove the two triangular estimates which will be used in
	Section \ref{secContinuous}.
	
	\begin{prop}
		\label{propTriangularContinuous}
		Let $m\ge1$.
		
		\begin{enumerate}[(i)]
			\item
			If $(T_t)_{t\ge0}$ is Kreiss bounded on a Hilbert space $H$, then
			\begin{equation*}\label{eqTriangularContinuousH}
				\norm{\mathcal K_m}_{
					B(L^2(J_m;H),L^2(I_m;H))}
				\le
				2eC_K(T)m.
			\end{equation*}
			
			\item
			If $(T_t)_{t\ge0}$ is positive and Kreiss bounded on $L^p(\Omega)$, then
			\begin{equation*}\label{eqTriangularPositiveContinuous}
				\norm{\mathcal K_m}_{
					B(L^p(J_m;L^p),L^p(I_m;L^p))}
				\le
				2eC_A(T)m.
			\end{equation*}
		\end{enumerate}
	\end{prop}
	
	\begin{proof}
		We first prove (i). Set $r :=\frac{1}{2m}$. By Proposition \ref{propHilbertMultipliersApp}(i),
		\begin{equation}\label{eqCrContinuous}
			\norm{\mathcal C_r}_{B(L^2(\R;H))} \le \frac{C_K(T)}{r}=2mC_K(T).
		\end{equation}
		Define
		\[
		(\mathcal K_m^{(r)}f)(s) := \int_{J_m} e^{-r(v-s)} T_{v-s}f(v)\,dv.
		\]
		If $\mathcal{E}_m:L^2(J_m;H)\to L^2(\mathbb R;H)$ denotes extension by zero and $\mathcal{P}_m:L^2(\mathbb R;H)\to L^2(I_m;H)$ restriction to $I_m$, then
		\begin{equation}\label{eqKmr=PmrCrEm}
			\mathcal K_m^{(r)} = \mathcal{P}_m\mathcal C_r \mathcal{E}_m.
		\end{equation}
		Indeed let $f\in L^2(J_m;H)$. For $s\in I_m$, we have
		\begin{align*}
			(\mathcal C_r\mathcal{E}_mf)(s)
			&= \int_0^\infty e^{-ru}T_u(\mathcal{E}_mf)(s+u)\,du\\
			&= \int_{\{u\ge0:\,s+u\in J_m\}} e^{-ru}T_u f(s+u)\,du.
		\end{align*}
		Since $s\in I_m$, the change of variables $v=s+u$ yields
		\[
		(\mathcal C_r\mathcal{E}_mf)(s) = \int_{J_m} e^{-r(v-s)}T_{v-s}f(v)\,dv = (\mathcal K_m^{(r)}f)(s),
		\]
		which gives \eqref{eqKmr=PmrCrEm}. Thus, by \eqref{eqCrContinuous}, $\norm{\mathcal K_m^{(r)}} \le 2mC_K(T)$.
		
		Define the multiplication operators $(D_I f)(s):=e^{-r s}f(s)$, $s\in I_m$, and $(D_J f)(v):=e^{-r v}f(v)$, $v\in J_m$. We have $\mathcal K_m = D_I\mathcal K_m^{(r)}D_J^{-1}$ and $\|D_I\|\,\|D_J^{-1}\| \le e^{r \cdot 2m}=e.$ Therefore
		\[
		\norm{\mathcal K_m} \le 2meC_K(T),
		\]
		which proves (i).
		\medskip
		
		We now prove (ii). Since the semigroup is positive and Kreiss bounded on $L^p$, its growth bound is nonpositive and the Laplace representation
		\begin{equation*}\label{eqLaplacepositive}
			R(r,-A)x
			=
			\int_0^\infty e^{-ru}T_ux\,du
		\end{equation*}
		holds for every $r>0$ and $x\in L^p(\Omega)$.
		
		Taking $r=\frac{1}{2m}$, we have for every $x\in L^p_+(\Omega)$,
		\[
		0
		\le
		\int_0^{2m}T_ux\,du
		\le
		e\int_0^{2m}e^{-u/(2m)}T_ux\,du
		\le
		eR\left(\frac{1}{2m},-A\right)x.
		\]
		Hence
		\begin{equation*}\label{eqCesaroPositiveContinuous}
			\norm{\int_0^{2m}T_u\,du}
			\le
			2eC_A(T)m.
		\end{equation*}
		
		Define the operator $(\mathcal V_{2m}f)(s) := \int_0^{2m}T_uf(s+u)\,du$. Applying Proposition \ref{prop:positive-convolution} (i) to the positive kernel $u\longmapsto \one_{[-2m,0]}(u)T_{-u}$ yields
		\begin{equation}\label{eqVLpositiveContinuous}
			\norm{\mathcal V_{2m}}_{B(L^p(\R;L^p))}
			\le
			2eC_A(T)m.
		\end{equation}
		
		Extend $f\in L^p(J_m;L^p)$ by zero outside $J_m$. For $s\in I_m$,
		\begin{equation}\label{eqFinalSec}
			(\mathcal V_{2m}f)(s)
			=
			\int_0^{2m}T_uf(s+u)\,du
			=
			\int_{J_m}T_{v-s}f(v)\,dv
			=
			(\mathcal K_mf)(s).
		\end{equation}
		Thus, by \eqref{eqVLpositiveContinuous},
		\[
		\norm{\mathcal K_m}
		\le
		2eC_A(T)m.
		\]
		This proves (ii).
	\end{proof}
	
	The second ingredient is an upper estimate for the orbit. In the Hilbertian
	case we use the following result from \cite{Arnold2022}, while in the positive case we use \cite[Proposition 3.2]{ArnoldCoine2023}.
	Although this result is stated there for $\sigma$-finite measure spaces, the
	$\sigma$-finiteness assumption can be removed by the standard reduction
	described in Section \ref{secIndividual}.

	\begin{prop}\label{propOrbitContinuous}
		\begin{enumerate}[(i)]
			\item
			If $(T_t)_{t\ge0}$ is Kreiss bounded on a Hilbert space $H$, then there
			exists $M_2>0$ such that
			\begin{equation*}\label{eqOrbitContinuousH}
				\left(
				\int_0^t\norm{T_sx}^2\,ds
				\right)^{1/2}
				\le
				M_2t\norm{x},
				\qquad t\ge1,
			\end{equation*}
			for every $x\in H$.
			
			\item
			If $(T_t)_{t\ge0}$ is positive and Kreiss bounded on $L^p(\Omega)$, then
			there exists $M_p>0$ such that
			\begin{equation*}\label{eqOrbitPositiveContinuous}
				\left(
				\int_0^t\norm{T_sx}_p^p\,ds
				\right)^{1/p}
				\le
				M_pt\norm{x}_p,
				\qquad t\ge1,
			\end{equation*}
			for every $x\in L^p(\Omega)$.
		\end{enumerate}
	\end{prop}
	
	Since $1<p<\infty$, the adjoint semigroup $(T_t^*)_{t\ge0}$ is a positive
	$C_0$-semigroup on $L^{p'}(\Omega)$. Moreover, since
	$R(r,-A^*)=R(r,-A)^*$ for $r>0$,
	\begin{equation}\label{eqAdjointAbelContinuous}
		C_A(T^*)=C_A(T).
	\end{equation}
	
	\subsection{Discrete time}\label{secPrelimDiscrete}
	
	Let $T\in B(X)$. For $m\in\N^*$, set
	\[
	I_m:=\{0,\ldots,m-1\},
	\qquad
	J_m:=\{m,\ldots,2m-1\},
	\]
	and define for $ f := (f_j)_{j\in J_m} \subset X$,
	\begin{equation*}\label{eqDefKmDiscrete}
		(K_mf)_i
		:=
		\sum_{j\in J_m}T^{j-i}f_j,
		\qquad i\in I_m.
	\end{equation*}
	
	\begin{prop}\label{propTriangularDiscrete}
		Let $m\ge1$.
		
		\begin{enumerate}[(i)]
			\item
			If $T$ is Kreiss bounded on a Hilbert space $H$, then
			\begin{equation*}\label{eqTriangularDiscreteH}
				\norm{K_m}_{
					B(\ell^2(J_m;H),\ell^2(I_m;H))}
				\le
				2eC_K(T)m.
			\end{equation*}
			
			\item
			If $T$ is positive and Kreiss bounded on $L^p(\Omega)$, then
			\begin{equation*}\label{eqTriangularPositiveDiscrete}
				\norm{K_m}_{
					B(\ell^p(J_m;L^p),\ell^p(I_m;L^p))}
				\le
				2eC_A(T)m.
			\end{equation*}
		\end{enumerate}
	\end{prop}
	
	\begin{proof}
		We begin with (i). Set $\rho:=1-\frac1{2m}$. By Proposition \ref{propHilbertMultipliersApp}(ii),
		\begin{equation*}\label{eqCrDiscreteBound}
			\norm{C_\rho}_{B(\ell^2(\Z;H))} \le \frac{C_K(T)}{1-\rho}=2mC_K(T).
		\end{equation*}
		Define $(K_m^{(\rho)}f)_i := \sum_{j\in J_m} \rho^{j-i}T^{j-i}f_j$, $i\in I_m$. If $E_m$ denotes extension by zero from $J_m$ to $\Z$ and $P_m$ restriction to $I_m$, then $K_m^{(\rho)} = P_mC_\rho E_m$. Therefore
		\[
		\norm{K_m^{(\rho)}} \le \frac{C_K(T)}{1-\rho} = 2mC_K(T).
		\]
		
		Define $(D_If)_i:=\rho^if_i$ and $(D_Jf)_j:=\rho^jf_j$. Then $K_m = D_IK_m^{(\rho)}D_J^{-1}$. Moreover,
		\[
		\norm{D_I}\norm{D_J^{-1}} \le \rho^{-(2m-1)} = \left(1-\frac1{2m}\right)^{-(2m-1)} \le e.
		\]
		Thus
		\[
		\norm{K_m} \le 2eC_K(T)m.
		\]
		\medskip
		
		We now prove (ii). For $m\ge1$, define
		\[
		(V_{2m}f)_i
		:=
		\sum_{n=0}^{2m-1}T^nf_{i+n},
		\qquad
		f\in \ell^p(\Z;L^p(\Omega)),
		\]
		and again set $\rho:=1-\frac1{2m}$. Since $T$ is positive and Abel bounded, for every
		$x\in L^p_+(\Omega)$,
		\begin{align*}
			0
			\le
			\sum_{n=0}^{2m-1}T^nx
			&\le
			\rho^{-(2m-1)}
			\sum_{n=0}^{2m-1}\rho^nT^nx\\
			&\le
			e(I-\rho T)^{-1}x.
		\end{align*}
		Hence
		\begin{equation*}\label{eqPartialSumPositiveDiscrete}
			\norm{\sum_{n=0}^{2m-1}T^n}
			\le
			\frac{eC_A(T)}{1-\rho}
			=
			2eC_A(T)m.
		\end{equation*}
		
		By the discrete positive convolution principle,
		\[
		\norm{V_{2m}}_{B(\ell^p(\Z;L^p))}
		\le
		2eC_A(T)m.
		\]
		
		Extend $f\in\ell^p(J_m;L^p)$ by zero outside $J_m$. For $i\in I_m$, setting $j=i+n$ gives
		\[
		( V_{2m}f)_i 
		= \sum_{n=0}^{2m-1}T^n f_{i+n} 
		= \sum_{j=i}^{i+2m-1} T^{j-i} f_j 
		= \sum_{j\in J_m} T^{j-i} f_j 
		= (K_mf)_i.
		\]
		We conclude that
		\[
		\norm{K_m}_{
			B(\ell^p(J_m;L^p),\ell^p(I_m;L^p))}
		\le
		2eC_A(T)m.
		\]
		This proves (ii).
	\end{proof}
	
	We next record the corresponding upper orbit estimates.
	
	\begin{prop}\label{propOrbitDiscrete}
		\begin{enumerate}[(i)]
			\item
			If $T$ is Kreiss bounded on a Hilbert space $H$, then there
			exists $M_2>0$ such that
			\begin{equation*}\label{eqOrbitDiscreteH}
				\left(\sum_{k=0}^{N}\norm{T^kx}^2\right)^{1/2}
				\le
				M_2(N+1)\norm{x}, 
				\qquad N\ge0
			\end{equation*}
			for every $x\in H$; see \cite{CohenCunyEisnerLin2020}.
			
			\item
			If $T$ is positive and Kreiss bounded on $L^p(\Omega)$, then there exists $M_p >0$ such that
			\begin{equation}\label{eqOrbitPositiveDiscrete}
				\left(
				\sum_{k=0}^{N}\norm{T^kx}_p^p
				\right)^{1/p}
				\le
				M_p(N+1)\norm{x}_p,
				\qquad N\ge0,
			\end{equation}
		\end{enumerate}
		for every $x\in L^p(\Omega)$.
	\end{prop}
	
	\begin{proof}
		We only provide the proof of (ii), since it seems to be absent from the literature. Let $N\ge1$, $\rho:=1-\frac1{N+1}$, and $x\in L^p(\Omega)$. By positivity, $|T^kx|\le T^k|x|$ for all $k\ge0$. Since $\rho^{-N}=\left(1+\frac1N\right)^N\le e$, the continuous embedding $\ell^1 \hookrightarrow \ell^p$ yields
		\[
		\left(\sum_{k=0}^{N}|T^kx|^p\right)^{1/p}
		\le
		\sum_{k=0}^{N}T^k|x|
		\le
		\rho^{-N}\sum_{k=0}^{N}\rho^kT^k|x|
		\le
		e\rho^{-1}R(\rho^{-1},T)|x|.
		\]
		Taking the $L^p$-norm and using the definition of $C_A(T)$, we get
		\begin{align*}
			\left(\sum_{k=0}^{N}\norm{T^kx}_p^p\right)^{1/p}
			&\le
			e\rho^{-1}\norm{R(\rho^{-1},T)}\norm{x}_p\\
			&\le
			e\,\frac{C_A(T)}{1-\rho}\norm{x}_p\\
			&=
			eC_A(T)(N+1)\norm{x}_p,
		\end{align*}
		which proves \eqref{eqOrbitPositiveDiscrete}.
	\end{proof}
	
	Finally, if $T$ is positive on $L^p(\Omega)$, then $T^*$ is positive on
	$L^{p'}(\Omega)$. Moreover, for every $\lambda>1$, $
	R(\lambda,T^*)=R(\lambda,T)^*.$ Hence
	\begin{equation}\label{eqAdjointAbelDiscrete}
		C_A(T^*)=C_A(T).
	\end{equation}
	\section{The continuous-time self-improvement}\label{secContinuous}
	
	We now isolate the reciprocal-weight argument. Once the triangular and upper orbit estimates are available, the remaining self-improvement argument is the same in the Hilbertian and positive $L^p$ settings.
	
	In this section, $I_m$ and $J_m$ refer to the intervals $[0,m]$ and $[m,2m]$, respectively.
	
	\begin{lemma}\label{lemContinuousSelfImprovement}
		Let $1<r<\infty$ and let $(T_t)_{t\ge0}$ be a $C_0$-semigroup on a Banach
		space $X$.  Assume that there exist $C_0,C_1>0$ such that, for every $m\ge1$,
		\begin{equation}\label{eqAbstractTriangularContinuous}
			\norm{\mathcal K_m}_{
				B(L^r(J_m;X),L^r(I_m;X))}
			\le
			2eC_0m,
		\end{equation}
		and, for every $t\ge1$ and $x\in X$,
		\begin{equation}\label{eqAbstractOrbitContinuous}
			\left(\int_0^t\norm{T_sx}^r\,ds\right)^{1/r}
			\le
			C_1t\norm{x}.
		\end{equation}
		Then there exists $C>0$ such that
		\begin{equation}\label{eqAbstractConclusionContinuous}
			\norm{T_t}
			\le
			C(1+t)^{1-\varepsilon_r(C_0)},
			\qquad t\ge0.
		\end{equation}
	\end{lemma}
	
	\begin{proof}
		Fix $t\ge2$ and $x\in X$. If $T_tx=0$, there is nothing to prove, so we
		assume that $T_tx\neq0$. For $0\le s\le t$, set
		\[
		z_s:=T_{t-s}x,
		\qquad
		w(s):=\norm{z_s}^r,
		\]
		and, for $0<m\le t$, define
		\[
		W(m):=\int_0^m w(s)\,ds.
		\]
		The function $w$ is continuous and strictly positive on $[0,t]$. Indeed,
		if $w(s)=0$ for some $s\in[0,t]$, then $z_s=0$ and the semigroup property
		would give
		\[
		T_tx=T_sT_{t-s}x=T_sz_s=0,
		\]
		which contradicts our assumption.
		
		Let $m\ge1$ be such that $2m\le t$. Since $w$ is continuous and strictly
		positive on the compact interval $J_m=[m,2m]$, the function
		\[
		f(v):=w(v)^{-1/(r-1)}z_v,
		\qquad v\in J_m,
		\]
		belongs to $L^r(J_m;X)$. Put
		\[
		A_m
		:=
		\int_m^{2m}w(v)^{-1/(r-1)}\,dv.
		\] Since $T_{v-s}z_v = z_s$ for $s\in I_m$ and $v\in J_m$, we have $(\mathcal K_mf)(s) = A_mz_s$.
		
		Thus \[\norm{\mathcal K_mf}_{L^r(I_m;X)}^r = A_m^rW(m) \quad  \text{and} \quad  \norm{f}_{L^r(J_m;X)}^r = A_m.\]
		Applying  \eqref{eqAbstractTriangularContinuous} yields
		\begin{equation*}\label{eqAbstractArContinuous}
			W(m)^{1/r}A_m^{1/r'} \le 2eC_0m, \qquad r'=r/(r-1).
		\end{equation*}
		Comparing this with Hölder's inequality 
		
		\begin{align*}
			m = \int_m^{2m} \frac{w(s)^{1/r}}{w(s)^{1/r}} \, ds 
			&\le \left( \int_m^{2m} w(s) \, ds \right)^{1/r} \left( \int_m^{2m} w(s)^{-1/(r-1)} \, ds \right)^{1/r'} \\
			&= \bigl(W(2m)-W(m)\bigr)^{1/r} A_m^{1/r'},
		\end{align*} we obtain
		\begin{equation}\label{eq:abstract-doubling-continuous}
			W(2m) \ge \left(1+(2eC_0)^{-r}\right)W(m).
		\end{equation}
		
		Setting $\theta_r:=(2eC_0)^{-r}$ and $\delta_r:=\log_2(1+\theta_r) = r\varepsilon_r(C_0)$, iterating \eqref{eq:abstract-doubling-continuous} gives
		\[
		W(2^\ell) \ge (1+\theta_r)^\ell W(1) = 2^{\ell\delta_r}W(1)
		\]
		for every integer $\ell\ge0$ such that $2^\ell\le t$.
		
		To relate $W(1)$ to $\norm{T_tx}$, let $D:=\sup_{0\le s\le1}\norm{T_s}<\infty$. Since $\norm{T_{t-s}x} \ge D^{-1}\norm{T_tx}$ for $s\in[0,1]$,
		we have \[
		W(1) \ge D^{-r}\norm{T_tx}^r.\]
		Choosing $\ell\ge0$ such that $M:=2^\ell \le t < 2M$, we obtain
		\[
		M^{\delta_r}D^{-r}\norm{T_tx}^r \le W(M) = \int_{t-M}^t \norm{T_ux}^r\,du \le C_1^rt^r\norm{x}^r,
		\]
		where we used estimate \eqref{eqAbstractOrbitContinuous}. Since $M>t/2$, it follows that for $t\ge 2$
		\[
		\norm{T_tx} \le 2^{\delta_r/r}DC_1 t^{1-\delta_r/r}\norm{x} = 2^{\varepsilon_r(C_0)}DC_1 t^{1-\varepsilon_r(C_0)}\norm{x}.
		\]
		Extending this bound to $t\ge 0$ by local boundedness concludes the proof of \eqref{eqAbstractConclusionContinuous}.
	\end{proof}
	
	We are now in a position to prove Theorem \ref{thmContinuousMain}.
	\begin{proof}[Proof of Theorem \ref{thmContinuousMain}]
		For (i), apply Lemma \ref{lemContinuousSelfImprovement} with
		\[
		r=2,
		\qquad
		C_0=C_K(T), \quad C_1 = M_2
		\]
		using Propositions \ref{propTriangularContinuous}(i) and
		\ref{propOrbitContinuous}(i).  This gives
		\eqref{eqMainHilbertContinuous}.
		
		For (ii), apply the same lemma with
		\[
		r=p,
		\qquad
		C_0=C_A(T),\quad C_1 = M_p
		\]
		using Propositions \ref{propTriangularContinuous}(ii) and
		\ref{propOrbitContinuous}(ii).  We obtain \[
		\norm{T_t}
		\le
		C(1+t)^{1-\varepsilon_p(C_A(T))}\]
		
		Applying the same argument to $(T_t^*)_{t\ge0}$, which is positive and Abel bounded with $C_A(T^*)=C_A(T)$ by \eqref{eqAdjointAbelContinuous}, gives the estimate with $p'$ in place of
		$p$.  For every $a\in(0,1)$, the map
		\[
		r\longmapsto \frac1r\log(1+a^r),\qquad r>0,
		\]
		is decreasing.  Hence the better exponent is $\varepsilon_q(C_A(T))$, where we recall
		$q=p\wedge p'$. 
	\end{proof}

	\section{The discrete self-improvement}\label{secDiscrete}
	
	The discrete-time proof is the exact analogue of the preceding one, with
	integrals replaced by sums.
	
	\begin{lemma}\label{lemDiscreteSelfImprovement}
		Let $1<r<\infty$ and let $T\in B(X)$.  Assume that there exist $C_0,C_1>0$
		such that, for every $m\ge1$,
		\begin{equation*}\label{eqAbstractTriangularDiscrete}
			\norm{K_m}_{
				B(\ell^r(J_m;X),\ell^r(I_m;X))}
			\le
			2eC_0m,
		\end{equation*}
		and, for every $N\ge0$ and $x\in X$,
		\begin{equation*}\label{eqAbstractOrbitDiscrete}
			\left(\sum_{k=0}^{N}\norm{T^kx}^r\right)^{1/r}
			\le
			C_1(N+1)\norm{x}.
		\end{equation*}
		Then
		\begin{equation*}\label{eqAbstractConclusionDiscrete}
			\norm{T^N}
			\le
			2^{\varepsilon_r(C_0)}C_1(N+1)^{1-\varepsilon_r(C_0)},
			\qquad N\ge0.
		\end{equation*}
	\end{lemma}
	
	\begin{proof}
		Fix $N\in\N$ and $x\in X$.  If $T^Nx=0$, there is nothing to prove.  Set
		\[
		z_j:=T^{N-j}x,
		\qquad
		w_j:=\norm{z_j}^r,
		\qquad
		W_m:=\sum_{j=0}^{m-1}w_j.
		\]
		All $w_j$ are strictly positive.  Let $2m\le N+1$ and define for $j\in J_m$
		\[
		f_j:=w_j^{-1/(r-1)}z_j,
		\qquad \text{and}
		\qquad
		A_m:=\sum_{j=m}^{2m-1}w_j^{-1/(r-1)}.
		\]
		Since $T^{j-i}z_j=z_i$ for $i\in I_m$ and $j\in J_m$,
		\[
		(K_mf)_i=A_mz_i,
		\qquad
		\norm{f}_{\ell^r(J_m;X)}^r=A_m.
		\]
		Hence
		\begin{equation*}\label{eqAbstractArDiscrete}
			W_m^{1/r}A_m^{1/r'}
			\le
			2eC_0m.
		\end{equation*}
		H\"older's inequality yields
		\[
		m
		\le
		(W_{2m}-W_m)^{1/r}A_m^{1/r'},
		\]
		and therefore
		\begin{equation*}\label{eqAbstractDoublingDiscrete}
			W_{2m}
			\ge
			\left(1+(2eC_0)^{-r}\right)W_m.
		\end{equation*}
		Let $
		\delta_r:=r\varepsilon_r(C_0).$ Iteration gives, whenever $2^\ell\le N+1$,
		\[
		W_{2^\ell}
		\ge
		2^{\ell\delta_r}W_1
		=
		2^{\ell\delta_r}\norm{T^Nx}^r.
		\]
		Choose $M=2^\ell$ with $M\le N+1<2M$.  Then
		\[
		M^{\delta_r}\norm{T^Nx}^r
		\le
		W_M
		\le
		\sum_{k=0}^{N}\norm{T^kx}^r
		\le
		C_1^r(N+1)^r\norm{x}^r.
		\]
		Since $M>(N+1)/2$, we obtain
		\[
		\norm{T^Nx}
		\le
		2^{\delta_r/r}C_1(N+1)^{1-\delta_r/r}\norm{x},
		\]
		which is the desired estimate.
	\end{proof}
	
	\begin{proof}[Proof of Theorem \ref{thmDiscreteMain}]
		For (i), apply Lemma \ref{lemDiscreteSelfImprovement} with
		\[
		r=2,
		\qquad
		C_0=C_K(T), \quad C_1 = M_2
		\]
		using Propositions \ref{propTriangularDiscrete}(i) and
		\ref{propOrbitDiscrete}(i).
		
		For (ii), apply the same lemma with
		\[
		r=p,
		\qquad
		C_0=C_A(T), \quad C_1 = M_p
		\]
		using Propositions \ref{propTriangularDiscrete}(ii) and
		\ref{propOrbitDiscrete}(ii).  Apply the result once more to $T^*$ and use
		\eqref{eqAdjointAbelDiscrete}; as in continuous time, the better of the
		$p$ and $p'$ exponents is $\varepsilon_q(C_A(T))$. 
	\end{proof}

	
	\section{Individually eventually positive $C_0$-semigroups on $L^p$-spaces}\label{secIndividual}
	
	We now consider an extension of the positive $L^p$ result to individually eventually positive semigroups. We deliberately treat this case separately from the positive case considered above. Indeed, for positive semigroups the positive convolution principle gives the required triangular estimate directly, 
	with an explicit constant, and leads to a shorter and quantitative proof. Under individual eventual positivity, there need not exist a uniform positivity time, and a different argument is therefore required. The proof relies on the domination results developed in
	\cite{ArnoldStability} and is therefore also of independent interest.

	If $E$ is a complex Banach lattice, we denote by $E_{\mathbb R}$
	its real part and by $E_+ $ its positive cone. For $f\in E_{\mathbb R}$,
	we write $f\ge0$ if $f\in E_+$, and $|f|$ denotes the lattice modulus of
	$f$. The same notation $|\cdot|$ is used for the modulus in the
	complexification of $E$.
	
	For $g\in E_+$, the principal ideal generated by $g$ is
	\[
	E_g
	:=
	\{f\in E:\ |f|\le c g
	\text{ for some }c\ge0\}.
	\]
	Equipped with the gauge norm
	\[
	\norm{f}_g
	:=
	\inf\{c\ge0:\ |f|\le c g\},
	\qquad f\in E_g,
	\]
	the space $E_g$ is itself a Banach lattice.
	
	We say that $(T_t)_{t\ge0}$ is \emph{individually eventually positive} if, for every $f\in E_+$, there exists $t_f\ge0$ such that \[ T_tf\ge0,\qquad t\ge t_f. \] We shall use the Ces\`aro constant 
	\[ C_{\mathrm{Ces}}(T) := \sup_{t\ge1} \frac1t\norm{\int_0^tT_s\,ds}. \]
	We recall that according to \cite[Proposition 2.6]{ArnoldCoine2023}, for individually 
	eventually positive $C_0$-semigroups, Cesàro boundedness and Kreiss boundedness are equivalent. Here Ces\`aro boundedness means that $C_{\mathrm{Ces}}(T) < \infty$. 
	
	To prove Theorem \ref{thmIndEv}
	We will use the following. 
	
	\begin{prop}\label{propIndividualDomination}
		Let $(T_t)_{t\ge0}$ be an individually eventually positive semigroup on a
		complex Banach lattice $E$. Let $Y$ be a real Banach space, let $g\in E_+$,
		and let
		\[
		J:Y\longrightarrow E_{\R}
		\]
		be bounded with $JY\subseteq E_g$. Then there exist $\tau\ge0$ and $c\ge0$
		such that, for every $t\ge\tau$ and $y\in Y$,
		\[
		|T_tJy|
		\le
		c\norm{y}_Y T_tg.
		\]
		In particular, $T_tg\ge0$ for every $t\ge\tau$.
	\end{prop} 
	\begin{proof} This is \cite[Corollary 2.2]{ArnoldStability}. \end{proof}
	
	From now on, for $1<p<\infty$, we denote $E = L^p(\Omega)$ so that $E_{\R} = L^p(\Omega, \R)$ and $E_{+} = L^p(\Omega, \R_+)$.
	
	We shall repeatedly use the following standard observation. If
	$U:\R\to L^p(\Omega)$ is strongly measurable, then its essential range is
	contained in a separable subspace of $L^p(\Omega)$ and hence, up to null
	sets, in $L^p(\Omega_0)$ for some $\sigma$-finite measurable subset
	$\Omega_0\subseteq\Omega$. Thus the usual Fubini-Tonelli arguments may be
	applied after this reduction. We shall use this observation without further
	comment; see also the proof of \cite[Lemma 1]{Vogt2022}.
	\begin{lemma}\label{lemIndividualLatticeDuality}
		Let
		$U\in L^p(\R;E)$ and $ H\in E_+$ such that
		\[
		\left|
		\int_{\R}h(s)U(s)\,ds
		\right|
		\le H
		\]
		for every $h\in L^{p'}(\R;\R)$ with $\norm{h}_{p'}\le1$. Then
		\begin{equation*}\label{eqIndividualLatticeDuality}
			\left(
			\int_{\R}|U(s)|^p\,ds
			\right)^{1/p}
			\le 2H.
		\end{equation*}
	\end{lemma}
	\begin{proof}
		After the $\sigma$-finite reduction described above, apply the real-valued
		argument of \cite[Lemma 1]{Vogt2022} to $\operatorname{Re}U$ and
		$\operatorname{Im}U$, and use
		$|U|\le|\operatorname{Re}U|+|\operatorname{Im}U|$.
	\end{proof}
	The specific ingredients needed to apply the self-improvement argument from Section \ref{secContinuous} are
	contained in the next proposition. 
	
	\begin{prop}\label{propIndividualEstimates} Assume that $(T_t)_{t\ge0}$ is individually eventually positive and Ces\`aro bounded on $E=L^p(\Omega)$. 
		Then there exist $\Gamma,\Lambda<\infty$ such that the following assertions hold. \begin{enumerate}[(i)] 
			\item For $L>0$, define $\mathcal V_L$  by $(\mathcal V_LF)(s) := \int_0^L T_uF(s+u)\,du$ for $F\in L^p(\R;E).$ Then for every $L \ge 1$, \begin{equation}\label{eqIndividualConvolution}
				\norm{\mathcal V_L}_{B(L^p(\R;E))} \le \Gamma L. 
			\end{equation} \item For every $R\ge1$ and $f\in E$, \begin{equation}\label{eqIndividualOrbit} \left( \int_0^R\norm{T_tf}_p^p\,dt \right)^{1/p} \le \Lambda R\norm{f}_p. \end{equation} \end{enumerate}
	\end{prop} 
	\begin{proof}
		For $a\ge0$, set $D_a:=\int_0^a\norm{T_u}\,du$. Fix $F\in L^p(\R;E_{\R})$ and set 
		\[g_F := \left( \int_{\R}|F(v)|^p\,dv \right)^{1/p} \in E_+.
		\]
		By Fubini-Tonelli, $\norm{g_F}_p=\norm{F}_{L^p(\R;E)}$. 
		
		Define $J_F:L^{p'}(\R;\R)\longrightarrow E_{\R}$ by $J_Fh:=\int_{\R}h(v)F(v)\,dv$. H\"older's inequality gives $|J_Fh| \le \norm{h}_{p'}g_F$, so $J_FL^{p'}(\R;\R)\subseteq E_{g_F}$. Proposition \ref{propIndividualDomination} yields $\tau_F,c_F\ge0$ such that for every $u\ge \tau_F$,
		\begin{equation}\label{eqIndividualJF}
			|T_uJ_Fh| \le c_F\norm{h}_{p'}T_ug_F.
		\end{equation}
		
		For $0\le a\le L$, define $(\mathcal V_{a,L}F)(s) := \int_a^L T_uF(s+u)\,du$. Let $L\ge\tau_F$ and $h\in L^{p'}(\R;\R)$. Fubini's theorem and \eqref{eqIndividualJF} give
		\[
		\left| \int_{\R}h(s)(\mathcal V_{\tau_F,L}F)(s)\,ds \right| \le c_F\norm{h}_{p'} \int_{\tau_F}^L T_ug_F\,du.
		\]
		Now Lemma \ref{lemIndividualLatticeDuality}  yields
		\[
		\norm{\mathcal V_{\tau_F,L}F}_{L^p(\R;E)} \le 2c_F \norm{\int_{\tau_F}^L T_ug_F\,du}_p.
		\]
		Consequently, for $L\ge\tau_F$,
		\[
		\norm{\mathcal V_LF}_{L^p(\R;E)} \le \Bigl( 2c_F C_{\mathrm{Ces}}(T)L + (2c_F+1)D_{\tau_F} \Bigr) \norm{F}_{L^p(\R;E)}.
		\]
		Since for $1\le L<\tau_F$, $\norm{\mathcal V_LF}_{L^p(\R;E)} \le D_{\tau_F}\norm{F}_{L^p(\R;E)}$, it follows that \[
		\sup_{L\ge1} \frac1L \norm{\mathcal V_LF}_{L^p(\R;E)} < \infty
		\]
		for every real $F\in L^p(\R;E)$, and hence for every complex $F$ by decomposition into real and imaginary parts. Banach-Steinhaus yields \eqref{eqIndividualConvolution}.
		\medskip 
		
		We now prove (ii). Fix $f\in E_+$. Choose $a\ge0$ such that for every $t\ge a$, $ T_tf\ge0$ and 
		define \[ F := \left( \int_a^{a+1}|T_rf|^p\,dr \right)^{1/p} \in E_+. \] Let $ Y:=L^{p'}([a,a+1];\R)$
		and $ J_fh:=\int_a^{a+1}h(r)T_rf\,dr. $  H\"older's inequality gives \[ |J_fh| \le \norm{h}_{p'}F. \] Proposition \ref{propIndividualDomination} therefore 
		yields $\tau_F,c_F\ge0$ such that for every $u \ge \tau_F$
		\begin{equation}\label{eqIndividualOrbitDomination} 
			|T_uJ_fh| \le c_F\norm{h}_{p'}T_uF. 
		\end{equation} 
		Set $ b:=a+\tau_F+1. $ Then $T_tf\ge0$ for every $t\ge b$. Fix $R>b$, and let \[ 0\le h\in L^{p'}([b,R]), \qquad \norm{h}_{p'}\le1, \] extended by zero outside $[b,R]$. For $u\ge\tau_F$, define $h_u(r):=h(u+r)$ for $r\in[a,a+1].$  Then $\norm{h_u}_{p'}\le1$. Using Fubini's theorem and the semigroup property gives \begin{align*}
			\int_b^R h(t) T_t f \, dt &= \int_b^R h(t) T_t f \left( \int_{t-a-1}^{t-a} du \right) dt \\
			&= \int_{b-a-1}^{R-a} \int_{u+a}^{u+a+1} h(t) T_t f \, dt \, du \\
			&= \int_{\tau_F}^{R-a} T_u \left( \int_a^{a+1} h_u(r) T_r f \, dr \right) du \\
			&= \int_{\tau_F}^{R-a} T_u J_f(h_u) \, du
		\end{align*}
		Since the vector on the left-hand side is positive, \eqref{eqIndividualOrbitDomination} implies \begin{equation}\label{eqIndividualOrbitOrder} 0 \le \int_b^R h(t)T_tf\,dt \le c_F\int_{\tau_F}^{R-a}T_uF\,du. \end{equation} 
		Hence \cite[Lemma 1]{Vogt2022} gives \[ \left( \int_b^R(T_tf)^p\,dt \right)^{1/p} = \sup_{\substack{ 0\le h\in L^{p'}([b,R])\\ \norm{h}_{p'}\le1 }} \int_b^R h(t)T_tf\,dt. \]
		It follows from \eqref{eqIndividualOrbitOrder} that 
		\[ \left( \int_b^R(T_tf)^p\,dt \right)^{1/p} \le c_F\int_{\tau_F}^{R-a}T_uF\,du. \]
		Taking $L^p$-norms and using Ces\`aro boundedness, we obtain 
		\begin{align*} \left( \int_b^R\norm{T_tf}_{L^p}^p\,dt \right)^{1/p} &\le c_F\norm{\int_{\tau_F}^{R-a}T_uF\,du}_{L^p}\\ &\le c_F\left( \norm{\int_0^{R-a}T_u\,du} + D_{\tau_F} \right)\norm{F}_{L^p}\\ &\le c_F\bigl(C_{\mathrm{Ces}}(T)R+D_{\tau_F}\bigr) \norm{F}_{L^p}. \end{align*}
		The integral over the fixed interval $[0,b]$ is controlled by local boundedness. Therefore \[ \sup_{R\ge1} \frac1R \left( \int_0^R\norm{T_tf}_p^p\,dt \right)^{1/p} <\infty \] 
		for every $f\in E_+$ and so for every $f \in E$. 
		Banach-Steinhaus therefore yields \eqref{eqIndividualOrbit}. \end{proof} 
	
	We can now prove Theorem \ref{thmIndEv}.
	
	\begin{proof}[Proof of Theorem \ref{thmIndEv}] Increase $\Gamma$, if necessary, so that $\Gamma\ge1$. For $m\ge1$,
		we recall that $I_m=[0,m]$,\, $J_m=[m,2m]$ and 
		for $f \in L^p(J_m, L^p)$, \[ (\mathcal K_mf)(s) = \int_{J_m}T_{v-s}f(v)\,dv, \qquad s\in I_m. \]
		From \eqref{eqFinalSec} and \eqref{eqIndividualConvolution} we deduce \[ \norm{\mathcal K_m}_{ B(L^p(J_m;E),L^p(I_m;E)) } \le 2\Gamma m. \] Together with \eqref{eqIndividualOrbit}, this is precisely the setting of Lemma \ref{lemContinuousSelfImprovement}, with \[ r=p, \qquad C_0=\frac{\Gamma}{e}, \qquad C_1=\Lambda. \] We therefore obtain \[ \norm{T_t} \le C(1+t)^{1-\varepsilon}, \qquad t\ge0, \] where \[ \varepsilon = \frac1p \log_2 \left( 1+(2\Gamma)^{-p} \right) >0. \] Since $\Gamma\ge1$, we have $\varepsilon<1$. This proves the assertion under Ces\`aro boundedness. The Kreiss bounded case follows immediately from \cite[Proposition 2.6]{ArnoldCoine2023}. \end{proof}
	
	\begin{remark} The preceding argument is non-quantitative. Indeed, the constant $\Gamma$ in \eqref{eqIndividualConvolution} is obtained by combining Proposition \ref{propIndividualDomination} from \cite{ArnoldStability} with
		Banach-Steinhaus, 
		and the proof does not provide a bound for $\Gamma$ in terms of $p$ and $C_K(T)$ (or $C_{\mathrm{Ces}}(T)$) alone. This contrasts with the positive case treated above, where the positive convolution principle yields an explicit triangular estimate and therefore an explicit exponent depending on the Abel constant. Moreover, the argument above directly gives the exponent associated with $p$. If the adjoint semigroup on $L^{p'}(\Omega)$ is also individually eventually positive (which is not automatic), the same proof can be applied to $(T_t^*)_{t\ge0}$, and one may then take the better of the two resulting exponents. \end{remark}

	\section{Kreiss bounded operators on UMD spaces}\label{secUMD}
	
	We now return to the discrete-time setting. We show that the positivity
	assumption can be removed on $L^p$-spaces and, more generally, that every
	Kreiss bounded operator on a UMD space has a polynomial gap below linear
	growth. The price for this greater generality is a less explicit exponent:
	unlike in the positive case, the exponent obtained below also depends on
	Fourier multiplier and decomposition constants of the underlying space.
	The argument is again based on a triangular estimate, but now in a Fourier
	norm.
	
	Let $X$ be a complex Banach space and let $1<r<\infty$. For a finite sequence
	$(x_j)_{j\in\Z}\subset X$, set
	\[
	\norm{(x_j)}_{\mathcal F_r(X)}
	:=
	\left(
	\int_{\T}\norm{\sum_j\gamma^{-j}x_j}_X^r\,d\gamma
	\right)^{1/r}.
	\]
	We first recall a standard consequence of the UMD property. If
	$a=(a_n)_{n\in\Z}$ is a bounded scalar sequence, the associated Fourier
	multiplier $T_a$ is initially defined on $X$-valued trigonometric polynomials
	by $T_a(\sum_n\gamma^{-n}x_n)=\sum_na_n\gamma^{-n}x_n$. Since $X$ is UMD, the
	vector-valued Marcinkiewicz multiplier theorem implies that there exists
	$M_{r,X}\ge1$ such that, whenever $a$ has bounded variation,
	\begin{equation}\label{eqBVMultiplier}
		\norm{T_a}_{B(L^r(\T;X))}
		\le
		M_{r,X}
		\left(
		\sup_{n\in\Z}|a_n|
		+
		\sum_{n\in\Z}|a_{n+1}-a_n|
		\right).
	\end{equation}
	See, for instance, \cite[Lemma 2.1]{DengLoristVeraar}.
	
	If $I\subseteq\Z$ is an interval, we denote by $P_I$ the Fourier projection
	onto the frequencies in $I$. Applying \eqref{eqBVMultiplier} to
	$a_n=\one_I(n)$ gives
	\begin{equation}\label{eqIntervalProjection}
		\norm{P_I}_{B(L^r(\T;X))}
		\le3M_{r,X}.
	\end{equation}
	
	We first record the upper orbit estimate which will be used below.
	
	\begin{lemma}\label{lemCunyBlock}
		Let $X$ be UMD, let $1<r<\infty$, and let $T\in B(X)$ be Kreiss bounded.
		Then, for all integers $0\le a\le b$ and every $x\in X$,
		\begin{equation}\label{eqCunyBlock}
			\left(
			\int_{\T}
			\norm{\sum_{n=a}^b\gamma^{-n}T^nx}^r\,d\gamma
			\right)^{1/r}
			\le
			B_{r,X}C_K(T)(b+1)\norm{x},
		\end{equation}
		where one may take $B_{r,X}:=18eM_{r,X}^3$.
	\end{lemma}
	
	\begin{proof}
		As shown in the proof of \cite[Theorem 3.1]{CunyLp},
		\[
		\left\|\sum_{n=0}^b\gamma^{-n}T^nx\right\|_{L^r(\T;X)}
		\le 6eM^2_{r,X}C_K(T)(b+1)\|x\|.
		\]
		Applying the uniformly bounded Fourier projection $P_{[a,b]}$ (see \eqref{eqIntervalProjection}) gives
		\eqref{eqCunyBlock} with $B_{r,X} = 18eM^3_{r,X}$. 
	\end{proof}
	
	The next lemma is the key point. We use the sets $I_m,J_m$ and the operator
	$K_m$ introduced in Section \ref{secPrelimDiscrete}.
	
	\begin{lemma}\label{lemTriangularFourier}
		Let $X$ be UMD, let $1<r<\infty$, and let $T\in B(X)$ be Kreiss bounded.
		Then, for every $m\ge1$ and $f = (f_j)_{j\in J_m} \subset X$,
		\begin{equation}\label{eqTriangularFourier}
			\norm{\bigl((K_mf)_i\bigr)_{i\in I_m}}_{\mathcal F_r(X)}
			\le
			A_{r,X}C_K(T)m
			\norm{(f_j)_{j\in J_m}}_{\mathcal F_r(X)},
		\end{equation}
		where one may take $A_{r,X}:=24eM_{r,X}^3$.
	\end{lemma}
	
	\begin{proof}
		Set $\rho:=1-\frac1{2m}$. For
		$(f_j)_{j\in J_m} \subset X$, put
		\[
		F(\gamma):=\sum_{j\in J_m}\gamma^{-j}f_j,
		\qquad
		G_\rho(\gamma):=(I-\rho\gamma T)^{-1}F(\gamma).
		\]
		Since $(I-\rho\gamma T)^{-1}
		=(\rho\gamma)^{-1}
		R((\rho\gamma)^{-1},T)$, the Kreiss condition gives
		\[
		\norm{G_\rho}_{L^r(\T;X)}
		\le2mC_K(T)\norm{F}_{L^r(\T;X)}.
		\]

		Expanding the resolvent, the Fourier coefficient of $G_\rho$ of index
		$i\in I_m$ is
		\[
		(K_m^{(\rho)}f)_i
		:=
		\sum_{j\in J_m}\rho^{j-i}T^{j-i}f_j.
		\]
		Thus $\sum_{i\in I_m}\gamma^{-i}(K_m^{(\rho)}f)_i=P_{I_m}G_\rho(\gamma)$,
		and \eqref{eqIntervalProjection} yields
		\[
		\norm{\bigl((K_m^{(\rho)}f)_i\bigr)_{i\in I_m}}_{\mathcal F_r(X)}
		\le
		6M_{r,X}C_K(T)m
		\norm{(f_j)_{j\in J_m}}_{\mathcal F_r(X)}.
		\]
		
		It remains to remove the weights. Set $g_j:=\rho^{-j}f_j$ for $j\in J_m$.
		Then, for every $i\in I_m$,
		\[
		(K_m^{(\rho)}g)_i
		=
		\sum_{j\in J_m}\rho^{j-i}T^{j-i}\rho^{-j}f_j
		=
		\rho^{-i}(K_mf)_i.
		\]
		
		Define the scalar sequences $\alpha=(\alpha_n)_{n\in\Z}$ and
		$\beta=(\beta_n)_{n\in\Z}$ by
		\[
		\alpha_n=
		\begin{cases}
			1, & n\le0,\\
			\rho^n, & 0\le n\le m-1,\\
			\rho^{m-1}, & n\ge m-1,
		\end{cases}
		\qquad
		\beta_n=
		\begin{cases}
			\rho^{-m}, & n\le m,\\
			\rho^{-n}, & m\le n\le2m-1,\\
			\rho^{-(2m-1)}, & n\ge2m-1.
		\end{cases}
		\]
		
		Hence \eqref{eqBVMultiplier} gives
		\[
		\norm{T_\alpha}_{B(L^r(\T;X))}
		\le 2M_{r,X},
		\qquad
		\norm{T_\beta}_{B(L^r(\T;X))}
		\le 2eM_{r,X}.
		\]
		
		Now
		\[
		\sum_{i\in I_m}\gamma^{-i}(K_mf)_i
		=
		T_\alpha\left(
		\sum_{i\in I_m}\gamma^{-i}(K_m^{(\rho)}g)_i
		\right),
		\qquad
		\sum_{j\in J_m}\gamma^{-j}g_j
		=
		T_\beta\left(
		\sum_{j\in J_m}\gamma^{-j}f_j
		\right).
		\]
		Combining these identities with the preceding weighted estimate yields
		\[
		\begin{aligned}
			\norm{\bigl((K_mf)_i\bigr)_{i\in I_m}}_{\mathcal F_r(X)}
			&\le
			2M_{r,X}\,6M_{r,X}C_K(T)m\,2eM_{r,X}
			\norm{(f_j)_{j\in J_m}}_{\mathcal F_r(X)}
			\\
			&=
			24eM_{r,X}^3C_K(T)m
			\norm{(f_j)_{j\in J_m}}_{\mathcal F_r(X)}.
		\end{aligned}
		\]
		Thus \eqref{eqTriangularFourier} holds with
		$A_{r,X}:=24eM_{r,X}^3$.
	\end{proof}
	
	We now combine the preceding estimate with lower Fourier decompositions.
	Recall that $X$ is said to have lower $\ell^q(L^r)$-Fourier decompositions
	if there exists $D\ge1$ such that, for every interval partition
	$\mathcal I$ of $\Z$ and every $X$-valued trigonometric polynomial $F$,
	\begin{equation}\label{eqLowerFourierDecomposition}
		\left(
		\sum_{I\in\mathcal I}
		\norm{P_IF}_{L^r(\T;X)}^q
		\right)^{1/q}
		\le
		D\norm{F}_{L^r(\T;X)}.
	\end{equation}
	By \cite[Theorem 2.8]{DengLoristVeraar}, for every $1<r<\infty$, a Banach
	space $X$ is UMD if and only if it has lower
	$\ell^q(L^r)$-Fourier decompositions for some $1<q<\infty$. We now prove Theorem~\ref{thmUMD} in the following more precise form.
	
	\begin{thm}\label{thmUMDPolynomialGap}
		Let $1<r,q<\infty$, let $X$ be a Banach space with lower
		$\ell^q(L^r)$-Fourier decompositions with constant $D$, and let
		$T\in B(X)$ be Kreiss bounded. Then there exists $C>0$ such that
		\[
		\norm{T^N}
		\le
		C(N+1)^{1-\varepsilon},
		\qquad N\ge0,
		\]
		where
		\[
		\varepsilon
		=
		\frac1q
		\log_2\left(
		1+\frac{1}{(A_{r,X}DC_K(T))^q}
		\right)>0.
		\]
		In particular, every Kreiss bounded operator on a UMD Banach space has a
		polynomial gap below linear growth.
	\end{thm}
	
	\begin{proof}
		Since the lower decomposition property implies that $X$ is UMD, Lemma
		\ref{lemTriangularFourier} applies. Fix $N\ge1$ and
		$x\in X$, and put $
		L:=\lfloor\log_2(N+1)\rfloor$ and for $0\le j\le2^L-1$, 
		\[
		z_j:=T^{N-j}x.
		\]
		
		For $0\le k\le L$ and $0\le l< L$ define 
		\[
		a_k
		:=
		\norm{(z_j)_{j\in I_{2^k}}}_{\mathcal F_r(X)}
		\quad \text{and} \quad 
		b_l
		:=
		\norm{(z_j)_{j\in J_{2^l}}}_{\mathcal F_r(X)}.
		\]
		
		For $0\le k\le L-1$ and $m:=2^k$, since
		$T^{j-i}z_j=z_i$ for $i\in I_m$ and $j\in J_m$, we have
		\[
		\sum_{j\in J_m}T^{j-i}z_j=mz_i,
		\qquad i\in I_m.
		\]
		Lemma \ref{lemTriangularFourier}, applied to
		$(z_j)_{j\in J_m}$, therefore gives for each $0\le k \le L-1 $,
		\begin{equation}\label{ineak}
			a_k\le A_{r,X}C_K(T) b_k.
		\end{equation}
		
		Now set $E_0 = \norm{T^Nx}^q$ and for $1 \le k \le L $
		\[
		E_k
		:=
		\norm{T^Nx}^q+\sum_{j=0}^{k-1}b_j^q.
		\]
		For $1\le k\le L$, the intervals $
		\{0\},J_1,J_2,\ldots,J_{2^{k-1}}
		$
		form a partition of $I_{2^k}$. Hence, after completing them to an interval
		partition of $\Z$, \eqref{eqLowerFourierDecomposition} applied to
		$
		F_k(\gamma)
		:=
		\sum_{j\in I_{2^k}}\gamma^{-j}z_j$,
		gives  for every $1\le k \le L$,
		\begin{equation}\label{ineEkak}
			E_k^{1/q}\le Da_k.
		\end{equation}
		Combining \eqref{ineak} and \eqref{ineEkak} 
		we obtain for every $0\le k \le L-1$, 
		\[
		b_k
		\ge
		\frac{1}{A_{r,X}DC_K(T)}E_k^{1/q},
		\]
		where for the case $k=0$ we used \eqref{ineak} and the fact that $D\ge 1$. Consequently, for every $0\le k \le L-1$,
		\[
		E_{k+1}
		=
		E_k+b_k^q
		\ge
		\left(
		1+\frac{1}{(A_{r,X}DC_K(T))^q}
		\right)E_k.
		\]
		
		Set $\eta:=(A_{r,X}DC_K(T))^{-q}$. Iterating the preceding inequality and using   \eqref{ineEkak} yields
		\begin{align*}
			D^qa_L^q \ge E_L
			\ge
			(1+\eta)^{L}E_0 
			=(1+\eta)^La_0^q =
			(1+\eta)^L\norm{T^Nx}^q.
		\end{align*}
		Since
		\[
		a_L
		=
		\left(
		\int_{\T}
		\norm{
			\sum_{j=0}^{2^L-1}\gamma^{-j}T^{N-j}x
		}^r\,d\gamma
		\right)^{1/r} = \left(
		\int_{\T}
		\norm{
			\sum_{n=N-2^L+1}^{N}\gamma^{-n}T^nx
		}^r\,d\gamma
		\right)^{1/r},
		\]
		Lemma \ref{lemCunyBlock} implies
		$a_L\le B_{r,X}C_K(T)(N+1)\norm{x}$, and hence
		\[
		(1+\eta)^{L/q}\norm{T^Nx}
		\le
		DB_{r,X}C_K(T)(N+1)\norm{x}.
		\]
		Since $
		(1+\eta)^{L/q}
		=
		(2^L)^{\frac1q\log_2(1+\eta)}$
		and $2^L>(N+1)/2$, we obtain
		\[
		\norm{T^Nx}
		\le
		C(N+1)^{1-\varepsilon}\norm{x},
		\]
		where
		\[
		\varepsilon
		=
		\frac1q\log_2(1+\eta)
		=
		\frac1q\log_2\left(
		1+\frac{1}{(A_{r,X}DC_K(T))^q}
		\right)>0,
		\]
		which proves the assertion.
	\end{proof}
	
	We finally specialize the result to $L^p$-spaces.
	
	\begin{cor}\label{corLpPolynomialGap}
		Assume that $(\Omega,\Sigma,\mu)$ is a $\sigma$-finite measure space, let
		$1<p<\infty$, and let $T\in B(L^p(\Omega))$ be Kreiss bounded. Then there
		exist $C>0$ and $\varepsilon>0$ such that
		\[
		\norm{T^N}_{B(L^p)}
		\le
		C(N+1)^{1-\varepsilon},
		\qquad N\ge0.
		\]
	\end{cor}
	
	\begin{proof}
		According to \cite[Example 2.12]{DengLoristVeraar}, $L^p(\Omega)$ has lower
		$\ell^q(L^p)$-Fourier decompositions for a suitable $1<q<\infty$. Hence
		Theorem \ref{thmUMDPolynomialGap} gives the result. Applying the same argument
		to $T^*$ on $L^{p'}(\Omega)$, using $C_K(T^*)=C_K(T)$ and
		$\norm{T^N}=\norm{(T^*)^N}$, yields a second exponent, and one may take the
		better of the two resulting exponents.
	\end{proof}

	\appendix
	\section{Proofs of the triangular estimates}\label{appTriangular}
	
	We collect here the Fourier multiplier arguments used in the Hilbertian parts of Propositions \ref{propTriangularContinuous} and \ref{propTriangularDiscrete}.

	\begin{prop}\label{propHilbertMultipliersApp}
		\begin{enumerate}[(i)]
			
			\item Let $(T_t)_{t\ge0}$ be a Kreiss bounded $C_0$-semigroup on a complex Hilbert space $H$ with generator $-A$. For $r>0$, the operator 
			\[
			(\mathcal C_r f)(s) := \int_0^\infty e^{-ru}T_u f(s+u)\,du \qquad (s\in\R),
			\]
			defined initially on $L^1(\R;H)\cap L^2(\R;H)$, extends uniquely to $\mathcal C_r\in B\bigl(L^2(\R;H)\bigr)$ with
			\begin{equation}\label{eq:continuous-Cr-bound-app}
				\norm{\mathcal C_r}_{B(L^2(\R;H))} \le \frac{C_K(T)}{r}.
			\end{equation}
			Moreover, for every $f\in L^2(\R;H)$ and almost every $\beta\in\R$,
			\begin{equation}\label{eqContinuousLMultiplierApp}
				\widehat{\mathcal C_r f}(\beta) = R(r-i\beta,-A)\widehat f(\beta).
			\end{equation}
			
			\item Let $T\in B(H)$ be Kreiss bounded and $0<\rho<1$. The operator
			\[
			(C_\rho f)_i := \sum_{n=0}^\infty \rho^nT^nf_{i+n} \qquad (i\in\Z),
			\]
			defined initially on $\ell^1(\Z;H)\cap\ell^2(\Z;H)$, extends uniquely to $C_\rho\in B\bigl(\ell^2(\Z;H)\bigr)$ with
			\begin{equation}\label{eq:discrete-Crho-bound-app}
				\norm{C_\rho}_{B(\ell^2(\Z;H))} \le \frac{C_K(T)}{1-\rho}.
			\end{equation}
			Moreover, for every $f\in\ell^2(\Z;H)$ and almost every $\gamma\in\T$ (with normalized Haar measure),
			\begin{equation}\label{eqDiscreteMultiplierApp}
				\widehat{C_\rho f}(\gamma) = (I-\rho\gamma T)^{-1}\widehat f(\gamma).
			\end{equation}
			
		\end{enumerate}
	\end{prop}

	\begin{proof}
		
		For (i), since $(T_t)_{t\ge0}$ is Kreiss bounded on $H$, for each $r>0$
		\[
		\int_0^\infty e^{-ru}\norm{T_u}\,du<\infty.
		\]
		For $f\in L^1(\R;H)\cap L^2(\R;H)$, Minkowski's inequality shows that $\mathcal C_rf\in L^2(\R;H)$. Moreover,
		\[
		\int_\R\int_0^\infty e^{-ru}\norm{T_uf(s+u)}\,du\,ds
		\le \norm{f}_{L^1(\R;H)}\int_0^\infty e^{-ru}\norm{T_u}\,du<\infty,
		\]
		so Fubini's theorem and the change of variables $v=s+u$ give
		\begin{align*}
			\widehat{\mathcal C_rf}(\beta)
			&=\int_0^\infty e^{-ru}T_u
			\left(\int_\R e^{-i\beta s}f(s+u)\,ds\right)du\\
			&=\left(\int_0^\infty e^{-(r-i\beta)u}T_u\,du\right)\widehat f(\beta)\\
			&=R(r-i\beta,-A)\widehat f(\beta).
		\end{align*}
		By the Fourier-Plancherel theorem and the Kreiss condition,
		\[
		\norm{\mathcal C_rf}_{L^2(\R;H)}
		\le \frac{C_K(T)}r\norm{f}_{L^2(\R;H)}.
		\]
		Density gives \eqref{eq:continuous-Cr-bound-app} and \eqref{eqContinuousLMultiplierApp} on all of $L^2(\R;H)$.
		
		\medskip
		
		For (ii), since $T$ is Kreiss bounded, for every $0<\rho <1$,
		\[
		\sum_{n=0}^\infty \rho^n\norm{T^n}<\infty.
		\]
		For $f\in \ell^1(\Z;H)\cap\ell^2(\Z;H)$,
		\[
		\sum_{i\in\Z}\sum_{n=0}^\infty \rho^n\norm{T^nf_{i+n}}
		\le \norm{f}_{\ell^1(\Z;H)}\sum_{n=0}^\infty \rho^n\norm{T^n}<\infty.
		\]
		Thus the sums may be interchanged and then
		\begin{align*}
			\widehat{C_\rho f}(\gamma)
			&=\sum_{n=0}^\infty \rho^nT^n
			\sum_{i\in\Z}f_{i+n}\gamma^{-i}=\sum_{n=0}^\infty(\rho\gamma T)^n\widehat f(\gamma)=(I-\rho\gamma T)^{-1}\widehat f(\gamma).
		\end{align*}
		If $\lambda=(\rho\gamma)^{-1}$, then
		$(I-\rho\gamma T)^{-1}=\lambda R(\lambda,T)$ and therefore
		\[
		\norm{(I-\rho\gamma T)^{-1}}
		\le \frac{C_K(T)}{1-\rho}.
		\]
		The Parseval theorem yields
		\[
		\norm{C_\rho f}_{\ell^2(\Z;H)}
		\le \frac{C_K(T)}{1-\rho}\norm{f}_{\ell^2(\Z;H)}.
		\]
		Since $\ell^1(\Z;H)\cap\ell^2(\Z;H)$ is dense in $\ell^2(\Z;H)$, this proves \eqref{eq:discrete-Crho-bound-app} and the multiplier identity \eqref{eqDiscreteMultiplierApp} extends by density.
	\end{proof}
	
	\section*{Declaration on the use of generative AI}
	
	During the preparation of this manuscript, the author used ChatGPT (GPT-5.6 Sol, OpenAI) for language editing, structural refinements, and exploratory mathematical discussions regarding proof strategies. All mathematical arguments and proofs were independently verified by the author, who takes full responsibility for the contents of the paper.

\end{document}